\documentclass{amsart}
\usepackage{amsmath,amssymb,amsthm,
verbatim}
\newtheorem{theorem}{Theorem}

\theoremstyle{definition}

\begin{document}
\title{On Conditions Relating to Nonsolvability}
\author{Michael~J.~J.~Barry}
\address{Department of Mathematics\\
Allegheny College\\
Meadville, PA 16335}
\email{mbarry@allegheny.edu}
\thanks{}


\begin{abstract}
Recent work of Kaplan and Levy refining a nonsolvability criterion proved by Thompson in his N-Groups paper prompts questions on whether certain conditions on groups are equivalent to nonsolvability.
\end{abstract}
\maketitle

In what follows, $G$ is a finite group with identity $1_G$ and $G^\#=G \setminus \{1_G\}$.

Thompson~\cite[Corollary 3]{T1968} proved the following : A finite group $G$ is nonsolvable if and only if there are three elements $x$, $y$, and $z$ in $G ^\#$, whose orders are coprime in pairs, such that $x y z=1_G$.

How much tighter can one make this nonsolvability criterion?  Can one always choose $x$, $y$, and $z$ to be elements of prime-power order, for distinct primes obviously?  Call a group that satisfies this condition a \emph{3PPO-group} (for three prime-power orders).  So is a group nonsolvable if and only if it is a 3PPO-group?  Can one always choose $x$, $y$, and $z$ to be elements of prime order?  Call a group that satisfies this condition a \emph{3PO-group}(for three prime orders).

In a recent paper~\cite{KL2010}, Kaplan and Levy show that $x$, $y$, and $z$ can be chosen so that $x$ has order a power of 2, $y$ has  order a power of $p$ for an odd prime $p$, and $z$ has order coprime to $2 p$.  In other words, two of the three elements can be chosen to have order a power of a prime.  In addition, they show that every nonabelian simple group is a 3PO-group.

In this short note, we show that not every nonsolvable group is a 3PO-group and we exhibit a condition equivalent to 3PPO.  

Our first result below shows $SL(2,5)$, the group of $2 \times 2$ matrices which entries in $GF(5)$ and determinant 1, is not a 3PO-group.  Since $SL(2,5)$ is a non-split extension of a central subgroup of order 2 by $A_5$, $SL(2,5)$ has the smallest possible order of a nonsolvable group that is not simple and does not contain a simple group as a subgroup.

\begin{theorem}
In $SL(2,5)$, there do not exist elements $x$, $y$, and $z$ in $SL(2,5)$ of distinct prime orders with $x y z=e$.
\end{theorem}

\begin{proof}
In this proof, we use the character table of $2 \cdot A_5 \cong SL(2,5)$ given on p. xxiv of~\cite{C1985} with its class labelings and its ordering of characters, which we label as $\chi_i$ with $1 \leq i \leq 9$.

Now the only possibility for three elements in $SL(2,5)$ to have distinct prime orders is for those orders to be 2, 3, and 5.  The group $SL(2,5)$ has one element of order 2, namely $-I_2$, whose conjugacy class is labeled $1A_1$.  In addition, $SL(2,5)$ has one conjugacy class of elements of order 3 labeled $3A_0$, and two conjugacy classes of elements of order 5 labeled $5A_0$ and $5B_0$.  Now denote $-I_2$ by $g_2$, an element of the conjugacy class $3A_0$ by $g_3$, and elements of the conjugacy classes $5A_0$ and $5B_0$ by $g_5$ and $h_5$, respectively.  Then 
\[\sum_{k=1}^9 \frac{1}{\chi_k(1_G)} \chi_k(g_2)\chi_k(g_3)\chi_k(g_5)
=1+0+0+(-1)+0+b_5+b^*_5+1+0,\]
where the $k$th term on the right-hand side is $\frac{1}{\chi_k(1_G)} \chi_k(g_2)\chi_k(g_3)\chi_k(g_5)$.

This right-hand side simplifies to
\[b_5+b^*_5+1=\frac{-1+\sqrt{5}}{2}+\frac{-1-\sqrt{5}}{2}+1=0.\]

Similarly
\[\sum_{k=1}^9 \frac{1}{\chi_k(1_G)} \chi_k(g_2)\chi_k(g_3)\chi_k(h_5)=0.\]

By~\cite[Lemma 19.2]{D1971}, these two calculations show that there are no elements $x$, $y$, and $z$ of order 2, 3, and 5, respectively, in $SL(2,5)$ such that $x y z =1_{SL(2,5)}$.
\end{proof}

We say that a group $G$ is a \emph{3SS-group} (for three Sylow subgroups) if and only if there are three Sylow subgroups $P_1$, $P_2$, and $P_3$ corresponding to three distinct primes $p_1$, $p_2$, and $p_3$ dividing $|G|$ such that $|P_1 P_2 P_3|<|P_1||P_2||P_3|$.  (Here $P_1P_2 P_3=\{x_1 x_2 x_3 \mid x_i \in P_i, 1 \leq i \leq 3\}$.)  Some time ago, Michael Ward and the present author tried unsuccessfully to prove that a group was nonsolvable if and only if it was a 3SS-group.
\begin{theorem}
A finite group $G$ is a 3PPO-group if and only if it it is a 3SS-group.
\end{theorem}

\begin{proof}
Suppose that $G$ is a 3PPO-group. Then there are three distinct primes $p_1$, $p_2$, and $p_3$ dividing $|G|$, and three elements $x_1$, $x_2$, and $x_3$ in $G^\#$, such that $x_i$ is a $p_i$-element for $i=1,2,3$ and $x_1 x_2 x_3=1_G$.  If, for $i=1,2,3$, $P_i$ is a Sylow $p_i$-subgroup containing $x_i$, then $|P_1 P_2 P_3|<|P_1||P_2||P_3|$, implying that $G$ is a 3SS-group.

Suppose that $G$ is a 3SS-group.  Then there are three Sylow subgroups $P_1$, $P_2$, and $P_3$ corresponding to three distinct primes $p_1$, $p_2$, and $p_3$ dividing $|G|$ such that $|P_1 P_2 P_3|<|P_1||P_2||P_3|$.  This implies that  there are distinct triples $(x_1,x_2,x_3)$ and $(y_1, y_2, y_3)$ in $P_1 \times P_2 \times P_3$ such that $x_1 x_2 x_3=y_1 y_2 y_3$, implying
\[(y_1^{-1} x_1)(x_2 y_2^{-1})(y_2 x_3 y_3^{-1} y_2^{-1})=1_G.\]

Since the triples are distinct, there is an $i$ with $1 \leq i \leq 3$ such that $x_i \neq y_i$.  From this it follows that for every $i$, $x_i \neq y_i$.  Thus $y_1^{-1} x_1$, $x_2 y_2^{-1}$, and $y_2 x_3 y_3^{-1} y_2^{-1}$ are non-trivial elements of prime-power order for three distinct primes, and this
implies that $G$ is a 3PPO-group.
\end{proof}

To our knowledge, the question of whether the condition 3PPO is equivalent to nonsolvability remains open.

\end{document}